\documentclass[11pt, a4paper]{amsart}

\usepackage{amsfonts,amsmath,amssymb, amscd,fullpage}
\usepackage{stmaryrd}
\usepackage[all]{xy}

\newtheorem{theorem}{Theorem}[section]
\newtheorem{lemma}[theorem]{Lemma}
\newtheorem{definition}[theorem]{Definition}
\newtheorem{proposition}[theorem]{Proposition}
\newtheorem{corollary}[theorem]{Corollary}

\theoremstyle{definition}

\newcommand\pf{\begin{proof}}
\newcommand\epf{\end{proof}}

\newcommand\C{\mathbb{C}}

\newcommand\B{\mathcal{B}}

\newcommand\coc{\mathcal C}
\newcommand\yd{\mathcal{YD}}
\newcommand\tor{\mathrm{Tor}}
\newcommand\ext{\mathrm{Ext}}

\DeclareMathOperator{\Hom}{Hom}

\DeclareMathOperator{\Ker}{Ker}

\DeclareMathOperator{\GL}{GL}
\DeclareMathOperator{\SL}{SL}

\numberwithin{equation}{section}

\title{Hochschild homology of Hopf algebras and  free Yetter-Drinfeld resolutions of the counit}

\author{Julien Bichon}
\address{
Laboratoire de Math\'ematiques,
Universit\'e Blaise Pascal,
Complexe universitaire des C\'ezeaux,
63177~Aubi\`ere Cedex, France}
\email{Julien.Bichon@math.univ-bpclermont.fr}


\begin{document}

\begin{abstract}
We show that if $A$ and $H$ are Hopf algebras that have equivalent tensor categories of comodules,
then one can transport what we call a  free Yetter-Drinfeld resolution of the counit of $A$ to 
the same kind of resolution for the counit of $H$, exhibiting in this way strong links between 
the Hochschild homologies of $A$ and $H$.
This enables us to get a finite free resolution of the counit of $\mathcal B(E)$, the Hopf algebra
of the bilinear form associated to an invertible matrix $E$, generalizing an ealier construction 
of Collins, H\"{a}rtel and Thom in the orthogonal case $E=I_n$. It follows that $\B(E)$ is smooth of dimension 3 and satisfies Poincar\'e duality. Combining this with results of Vergnioux,
it  also follows that when $E$ is an antisymetric matrix, the $L^2$-Betti numbers of the associated
discrete quantum group all vanish. We also use our resolution to compute the bialgebra cohomology
of $\B(E)$ in the cosemisimple case.
\end{abstract}

\maketitle

\section{introduction}

Let $n \in \mathbb N^*$ and let $A_o(n)$ be the algebra (over the field of complex numbers) presented by generators $(u_{ij})_{1 \leq i,j\leq n}$ and relations making the matrix $u=(u_{ij})$ orthogonal. This is a Hopf algebra, introduced by Dubois-Violette and Launer \cite{dvl} and independently by Wang \cite{wa95} in the compact quantum group setting. The Hopf algebras $A_o(n)$ play an important role in quantum group theory, since any finitely generated Hopf algebra of Kac type (the square of the antipode is the identity), and in particular any semisimple Hopf algebra, is a quotient of one of these. They have been studied from several perspectives, in particular from  the (co)representation theory viewpoint \cite{ba96,bi1,bideva} and the probabilistic and operator algebraic viewpoint
\cite{baco,bacozi,vaever,vava,voi11}.

The homological study of $A_o(n)$ begins in \cite{cht}, where Collins, H\"artel and Thom define an exact sequence of $A_o(n)$-modules
$$0 \rightarrow A_o(n) \longrightarrow
  M_n(A_o(n)) \longrightarrow
 M_n(A_o(n)) \longrightarrow A_o(n) \overset{\varepsilon} \longrightarrow \mathbb C \to 0 \quad (\star)$$
thus yielding a resolution of the conit of $A_o(n)$ by free $A_o(n)$-modules. From this exact sequence,
they deduce some important homological information on $A_o(n)$:
\begin{enumerate}
 \item $A_o(n)$ is smooth of dimension 3,
\item $A_o(n)$ satisfies Poincar\'e duality,
\item The $L^2$-Betti numbers of $A_o(n)$ all vanish.
\end{enumerate}
An inconvenient in \cite{cht} is that the verification of the exactness of $\star$ is a very long computation involving tedious Gr\"obner basis computations. It is the aim of the present paper to propose a simpler and more conceptual proof of the exactness of the sequence $\star$, together with a generalization to a larger class of Hopf algebras.  

Our starting point is the combination of the following two known facts.
\begin{enumerate}
 \item For $q \in \mathbb C^*$, there exists a resolution of the counit of $\mathcal O(SL_q(2))$
having the same length as the one of the sequence $\star$ (see e.g. \cite{hk}).
\item For $q$ satisfying $q +q^{-1}= -n$, there exists an equivalence of tensor categories of comodules $\mathcal M^{\mathcal O(SL_q(2))} \simeq^\otimes \mathcal M^{A_o(n)}$ \cite{bi1}. 
\end{enumerate}
Therefore, although one cannot expect that a tensor equivalence between categories of comodules induces isomorphisms between Hochschild homologies, it is tempting to believe that it is possible to use the above monoidal equivalence to transport a resolution of the counit of 
 $\mathcal O(\SL_q(2))$ by free modules having appropriate additional structures (in particular a comodule structure)
to get a resolution of the counit of $A_o(n)$ having the same length.

The appropriate structure we find is that of free Yetter-Drinfeld module, see Section 3 for the definition, these are  Yetter-Drinfeld modules that are in particular free as modules.
We show that if $A$ and $H$ are Hopf algebras that have equivalent tensor categories of comodules,
then one can transport a free  Yetter-Drinfeld resolution of the counit of $A$ to 
the same kind of resolution for the counit of $H$ (with preservation of the length of the resolution).

Now let $E \in \GL_n(\C)$ with $n \geq 2$ and consider the algebra $\mathcal B(E)$ presented by generators 
$(u_{ij})_{1 \leq i,j\leq n}$ and relations
$$E^{-\!1} u^t E u = I_n = u E^{-\!1} u^t E,$$
where $u$ is the matrix $(u_{ij})_{1 \leq i,j \leq n}$. The Hopf algebra $\B(E)$ was defined in \cite{dvl}, and corresponds to the quantum symmetry group of the bilinear form associated to $E$. We have $\B(I_n) = A_o(n)$ and  $\mathcal O(\SL_q(2)) = \B(E_q)$, where 
$$E_q =  \left(\begin{array}{cc} 0 & 1 \\
                          -q^{-1} & 0\\
       \end{array} \right)$$
We construct, for any $E \in \GL_n(\C)$, an exact sequence of $\B(E)$-modules
$$0 \rightarrow \B(E) \longrightarrow
  M_n(\B(E)) \longrightarrow
 M_n(\B(E)) \longrightarrow \B(E) \overset{\varepsilon} \longrightarrow \mathbb C \to 0 \quad (\star_E)$$
See Section 5. For $E=I_n$, the sequence is the one of Collins-H\"artel-Thom in \cite{cht}. The verification of exactness goes as follows.
\begin{enumerate}
\item  We endow each constituent of the sequence of a free Yetter-Drinfeld module structure.
 \item We use the previous construction to transport sequences of free Yetter-Drinfeld modules to show that for $E\in \GL_n(\C)$, $F \in \GL_m(\C)$ with ${\rm tr}(E^{-1}E^t)={\rm tr}(F^{-1}F^t)$ and $m,n\geq 2$ (so that $\mathcal M^{\B(E)} \simeq^\otimes \mathcal M^{\B(F)}$, \cite{bi1}), the sequence $\star_E$ is exact if and only if the sequence $\star_F$ is exact. 
\item We check that for any $q \in \C^*$, the sequence $\star_{E_q}$ is exact (this is less than a one page computation). Now for any $E \in \GL_n(\C)$ with $n \geq 2$, we pick $q \in \C^*$ such that ${\rm tr}(E^{-1}E^t)=-q-q^{-1}={\rm tr}(E_q^{-1}E_q^t)$, and we conclude from the previous item that $\star_E$ is exact.
\end{enumerate}
Similarly as in \cite{cht}, the exactness of the sequence $\star_E$ has several interesting consequences.
The first one is that $\B(E)$ is smooth of dimension $3$ for any $E \in \GL_n(\C)$, $n \geq 2$ (recall \cite{vdb} that an algebra $A$ is said to be smooth of dimension $d$ is the $A$-bimodule $A$ has a finite resolution of length $d$ by finitely generated projective $A$-bimodules, with $d$ being the smallest possible length for such a resolution).

The second consequence is that $\B(E)$ satisfies a Poincar\'e duality between its Hochschild homology and cohomology.
Since Van den Bergh's seminal paper \cite{vdb}, Poincar\'e duality for algebras has been the subject of many papers, in which the authors propose axioms that will have Poincar\'e duality as a corollary, see e.g. \cite{kokr} for a recent general and powerful framework. Let us emphasize that the exact sequence $\star_E$ enables us to establish Poincar\'e duality in a straightforward manner, without having to check any   condition such as the ones proposed in \cite{bz} (where moreover noetherianity assumptions were done, while $\B(E)$ is not noetherian if $n\geq 3$).   

A third consequence concerns bialgebra cohomology (Gerstenhaber-Schack cohomology \cite{gs1,gs2}), for which only very few full computations are known for non-commutative and non-cocommutative Hopf algebras  (see \cite{tai07}).
The fact that the exact sequence $\star_E$ consists of free Yetter-Drinfeld modules enables us to compute the bialgebra cohomology 
of $\B(E)$ (and hence in particular of $\mathcal O(\SL_q(2))$) in the cosemisimple case.

The last consequence concerns $L^2$-Betti numbers. Recall that the definition of $L^2$-Betti numbers for groups \cite{luckbook} can be generalized to discrete quantum groups (compact Hopf algebras of Kac type) \cite{ky08}. Combining the results in \cite{cht} with a result of Vergnioux \cite{ver12} (vanishing of the first Betti number of $A_o(n)$), Collins, H\"artel and Thom have shown that the $L^2$-Betti numbers of $A_o(n)$ all vanish. Using similar arguments we show that the $L^2$-Betti numbers of $A_o(J_m)$ all vanish, where $J_m$ is any anti-symmetric matrix. This completes the computation of the $L^2$-Betti numbers of all universal orthogonal discrete quantum groups of Kac type.  

The paper is organized as follows.
Section 2 is devoted to preliminaries. 
In Section 3 we introduce free Yetter-Drinfeld modules and remark that the standard resolution of the counit of a Hopf algebra is a free Yetter-Drinfeld resolution. In Section 4 we show how to transport free Yetter Drinfeld resolutions for Hopf algebras having equivalent tensor categories of comodules. In Section 5 we define and prove the exactness of the announced resolution of the counit of $\B(E)$. Section 6 is devoted to the several applications
we have already announced, and Section 7 consists of concluding remarks.

\section{Preliminaries}

\subsection{Notations and conventions} We assume that the base field is $\C$, the field of complex numbers (although our results, except in Subsection 6.4, do not depend on the base field). We assume that the reader is familiar with the theory of Hopf algebras and their tensor categories of comodules, as e.g. in \cite{kas,ks,mon}.
If $A$ is a Hopf algebra, as usual, $\Delta$, $\varepsilon$ and $S$ stand respectively for the comultiplication, counit and antipode of $A$. We use Sweedler's notations in the standard way. The category of right $A$-comodules is denoted $\mathcal M^A$.  If $M$ is an $A$-bimodule, then $H_*(A,M)$ and $H^*(A,M)$ denote  the respective Hochschild homology and cohomology groups of $A$ (with coefficients in $M$). 

\subsection{Hochschild homology of Hopf algebras and resolutions of the counit}

In this section we recall how the Hochschild homology and cohomology 
of a Hopf algebra $A$ can be described by using suitable $\tor$ and $\ext$ groups on the 
category of left or right $A$-modules and resolutions of the counit. 
This has been discussed under various forms in several papers (see \cite{ft}, \cite{gk}, \cite{hk}, \cite{bz}, \cite{cht})
and probably has its origins in \cite{ce}, Section 6 of chapter X.

\begin{proposition}\label{ident}
 Let $A$ be a Hopf algebra and let $M$ be an $A$-bimodule.
 Define a left $A$-module structure on $M$ and a right $A$-module structure on $M$ by
 $$ a \rightarrow x = a_{(2)}\cdot x \cdot S(a_{(1)}), \quad
 x \leftarrow a = S(a_{(1)})\cdot x \cdot a_{(2)}$$
 and denote by $M'$ and $M''$ the respective corresponding left $A$-module and 
 right $A$-module. 
 Then for all $n \in \mathbb N$ there exist isomorphisms of
 vector spaces
 $$H_n(A,M) \simeq \tor^A_n(\C_\varepsilon, M'), \quad H^n(A,M) \simeq \ext^n_A(\C_\varepsilon, M'')$$
\end{proposition}

The previous $\ext$-groups are those in the category of right $A$-modules.
In lack of a reference that would give exactly the isomorphisms of Proposition \ref{ident},
we will provide an explicit proof. Writing down the proof also gives us the opportunity
to review
the  material involved in the statement of the proposition \cite{wei}.

Let $(A, \varepsilon)$ be an augmented algebra, i.e. $A$ is an algebra and $\varepsilon : A \longrightarrow \C$
is an algebra map that we call the counit of $A$.
We view $\C$ as a right $A$-module via $\varepsilon$ and we denote by $\C_\varepsilon$ this 
right $A$-module. 
Recall that the standard resolution of $\C_\varepsilon$ (the standard resolution of the counit)
is given by the complex of free right $A$-modules
$$\cdots \longrightarrow  A^{\otimes n+1} \longrightarrow A^{\otimes n} \longrightarrow \cdots 
\longrightarrow A \otimes A \longrightarrow A \longrightarrow 0 $$
where each differential is given by
\begin{align*}
 A^{\otimes n+1} &\longrightarrow A^{\otimes n} \\
 a_1 \otimes \cdots \otimes a_{n+1} &\longmapsto \varepsilon(a_1) a_2 \otimes \cdots \otimes a_{n+1} + 
 \sum_{i=1}^n(-1)^i a_1 \otimes \cdots \otimes a_i a_{i+1} \otimes \cdots \otimes a_{n+1}
\end{align*}
Given a left $A$-module $M$, the vector spaces $\tor^A_*(\C_\varepsilon, M)$
are given by the homology of the complex obtained by tensoring any projective resolution of $\C_\varepsilon$
by $- \otimes_A M$. Thus using
the standard resolution of the counit, after suitable identifications,
we see that the vector spaces $\tor^A_*(\C_\varepsilon, M)$
are given by the homology of the following complex
$$\cdots \longrightarrow  A^{\otimes n} \otimes M \overset{d}\longrightarrow A^{\otimes n-1}\otimes M \overset{d} \longrightarrow \cdots \overset{d}
\longrightarrow A \otimes M \overset{d} \longrightarrow M \longrightarrow 0 \quad $$
where the differential $d :  A^{\otimes n} \otimes M \longrightarrow A^{\otimes n-1}\otimes M$ is given by
\begin{align*}
d(a_1 \otimes \cdots \otimes a_n  \otimes x)
=&\varepsilon(a_1) a_2 \otimes \cdots \otimes a_{n} \otimes x + 
 \sum_{i=1}^{n-1}(-1)^i a_1 \otimes \cdots \otimes a_i a_{i+1} \otimes \cdots \otimes a_{n} \otimes x \\
&+ (-1)^na_1 \otimes \cdots \otimes a_{n-1} \otimes a_n\cdot x
\end{align*}
Recall now that if $A$ is an algebra and $M$ is an $A$-bimodule, the Hochschild homology
groups $H_*(A,M)$ are the homology groups of the complex
$$\cdots \longrightarrow  M \otimes A^{\otimes n}  \overset{b}\longrightarrow M \otimes A^{\otimes n-1} \overset{b}\longrightarrow \cdots 
\overset{b}\longrightarrow M \otimes A \overset{b}\longrightarrow M \longrightarrow 0 \quad $$
where the differential $b :  M \otimes A^{\otimes n}  \longrightarrow M \otimes A^{\otimes n-1}$ is given by
\begin{align*}
b(x\otimes a_1 \otimes \cdots \otimes a_n  )
=&x \cdot a_1 \otimes \cdots \otimes a_{n} + 
 \sum_{i=1}^{n-1}(-1)^i x \otimes a_1 \otimes \cdots \otimes a_i a_{i+1} \otimes \cdots \otimes a_{n} \\
&+ (-1)^n a_n\cdot x \otimes a_1 \otimes \cdots \otimes a_{n-1}
\end{align*}
Assume now that $A$ is a Hopf algebra and let $M$ be an $A$-bimodule.
Consider the linear map
\begin{align*}
 \theta : M \otimes A^{\otimes n} &\longrightarrow A^{\otimes n} \otimes M' \\
  x \otimes a_1 \otimes \cdots \otimes a_n &\longmapsto
  a_{1(2)} \otimes \cdots \otimes a_{n(2)} \otimes x\cdot(a_{1(1)}  \cdots  \cdots a_{n(1)})
 \end{align*}
 It is straightforward to see that $\theta$ is an isomorphism
 with inverse given by
 \begin{align*}
  \theta^{-1} : A^{\otimes n} \otimes M'&\longrightarrow M \otimes A^{\otimes n} \\
  a_1 \otimes \cdots \otimes a_n \otimes x &\longmapsto
  x \cdot S(a_{1(1)}  \cdots  a_{n(1)}) \otimes (a_{1(2)} \otimes  \cdots \otimes \cdots a_{n(2)})
 \end{align*}
 and that $d \circ \theta = \theta \circ b$. Hence $\theta$ induces an isomorphism between the complexes defining $H_*(A,M)$ and $\tor^A_*(\C_\varepsilon, M')$
 and we get the first isomorphism $H_*(A,M) \simeq \tor^A_*(\C_\varepsilon, M')$.

 \medskip
 
For the second isomorphism in Proposition \ref{ident}, let us come back to the situation
of an augmented algebra $(A, \varepsilon)$.
Given a right $A$-module $M$, the vector spaces $\ext^*_A(\C_\varepsilon, M)$
are given by the cohomology of the complex obtained by applying the functor 
${\rm Hom}_A(-,M)$ to any projective resolution of $\C_\varepsilon$. Thus, using
the standard resolution of the counit, we see that after suitable identifications,
the vector spaces $\ext_A^*(\C_\varepsilon, M)$
are given by the cohomology of the following complex
$$ 0 \longrightarrow {\rm Hom}(\C,M) \overset{\partial}\longrightarrow \Hom(A,M) \overset{\partial}\longrightarrow
\cdots \overset{\partial}\longrightarrow\Hom(A^{\otimes n}, M) \overset{\partial}\longrightarrow \Hom(A^{\otimes n+1}, M) \overset{\partial}\longrightarrow \cdots$$
where the differential $\partial : \Hom(A^{\otimes n}, M) \longrightarrow \Hom(A^{\otimes n+1}, M)$
is given by 
\begin{align*}
\partial(f)(a_1 \otimes \cdots \otimes a_{n+1}) = &\varepsilon(a_1)
f(a_2 \otimes \cdots \otimes a_{n+1}) + \sum_{i=1}^{n}(-1)^i f(a_1 \otimes \cdots \otimes a_i a_{i+1} \otimes \cdots \otimes a_{n+1}) \\
&+ (-1)^{n+1} f(a_1 \otimes \cdots \otimes a_{n}) \cdot a_{n+1}
\end{align*}
If $A$ is an algebra and $M$ is an $A$-bimodule, the Hochschild cohomology
groups $H^*(A,M)$ are the cohomology groups of the complex
$$ 0 \longrightarrow {\rm Hom}(\C,M) \overset{\delta}\longrightarrow \Hom(A,M) \overset{\delta}\longrightarrow
\cdots \overset{\delta}\longrightarrow\Hom(A^{\otimes n}, M) \overset{\delta}\longrightarrow \Hom(A^{\otimes n+1}, M) \overset{\delta}\longrightarrow \cdots$$
where the differential $\delta : \Hom(A^{\otimes n}, M) \longrightarrow \Hom(A^{\otimes n+1}, M)$
is given by 
\begin{align*}
\delta(f)(a_1 \otimes \cdots \otimes a_{n+1}) = &a_1\cdot
f(a_2 \otimes \cdots \otimes a_{n+1}) + \sum_{i=1}^{n}(-1)^i f(a_1 \otimes \cdots \otimes a_i a_{i+1} \otimes \cdots \otimes a_{n+1}) \\
&+ (-1)^{n+1} f(a_1 \otimes \cdots \otimes a_{n}) \cdot a_{n+1}
\end{align*}
Assume now that $A$ is a Hopf algebra and let $M$ be an $A$-bimodule.
Consider the linear map
\begin{align*}
 \vartheta : \Hom(A^{\otimes n}, M) &\longrightarrow \Hom(A^{ \otimes n}, M'') \\
 f &\longmapsto \hat{f}, \ \hat{f}(a_1 \otimes \cdots \otimes a_n) = S(a_{1(1)} \cdots a_{n(1)})
 f(a_{1(2)} \otimes \cdots \otimes a_{n(2)})
\end{align*}
It is easy to see that $\vartheta$ is an isomorphism and that $\partial \circ \vartheta
= \vartheta \circ \delta$. Hence $\vartheta$ induces an isomorphism between the complexes defining $H^*(A,M)$ and $\ext_A^*(\C_\varepsilon, M'')$
 and we get the second isomorphism $H^*(A,M) \simeq \ext_A^*(\C_\varepsilon, M'')$. 

%\begin{proof}[Proof of Proposition \ref{ident}]
 %\begin{align*}
  %M \otimes A^{\otimes n} &\longrightarrow A^{\otimes n} \otimes M \\
  %x \otimes a_1 \otimes \cdots \otimes a_n &\longmapsto
  %a_{1(2)} \otimes \cdots \otimes a_{n(2)} \otimes x\cdot(a_{1(1)}  \cdots  \cdots a_{n(1)})
 %\end{align*}
%\end{proof}

\subsection{The Hopf algebra $\B(E)$}
 Let $E \in \GL_n(\C)$. Recall that the algebra $\mathcal B(E)$ \cite{dvl} is presented by generators 
$(u_{ij})_{1 \leq i,j\leq n}$ and relations
$$E^{-\!1} u^t E u = I_n = u E^{-\!1} u^t E,$$
where $u$ is the matrix $(u_{ij})_{1 \leq i,j \leq n}$. It has a Hopf algebra structure defined by
$$\Delta(u_{ij})
= \sum_{k=1}^n u_{ik} \otimes u_{kj}, \ 
\varepsilon(u_{ij}) = \delta_{ij}, \ 
S(u) = E^{-1}u^t E$$
For the matrix $E_q \in \GL_2(\C)$ in the introduction, we have $\B(E_q) = \mathcal O(\SL_q(2))$, and thus the Hopf algebras $\B(E)$ are generalizations of $\mathcal O(\SL_q(2))$. It is shown in \cite{bi1} that 
the isomorphism class of the Hopf algebra $\B(E)$ only depends on the bilinear form associated to the matrix $E$,
and that 
for $q \in \C^*$ satisfying  ${\rm tr}(E^{-1}E^t)= -q-q^{-1}$, the tensor categories of comodules over $\B(E)$ and $\mathcal O(\SL_q(2))$ are equivalent. 

The fundamental $n$-dimensional $\mathcal B(E)$-comodule is denoted by $V_E$: it has a basis
$e_1^E, \ldots , e_n^E$ and right coaction $\alpha : V_E \rightarrow V_E \otimes \B(E)$ defined by
$\alpha(e_i^E) =\sum_{k=1}^n e_k^E \otimes u_{ki}$. For future use, we record that the following linear maps
\begin{align*}
\delta : \C &\longrightarrow V_E \otimes V_E, \  &\varphi :  V_E &\longrightarrow V_E^*\\
1 & \longmapsto \sum_{i,j=1}^n E_{ij}^{-1} e_i^E \otimes e_j^E, \
& e_i^E &\longmapsto \sum_{k=1}^n E_{ik}e_k^{E*}
\end{align*}
are morphisms of $\B(E)$-comodules (where $E^{-1}=(E_{ij}^{-1})$).

We now define some maps that will be used in Section 6.
\begin{enumerate}
\item The sovereign character of $\B(E)$ is the algebra map $\Phi : \B(E) \rightarrow \C$ defined by
$\Phi(u)= E^{-1}E^t$. It satisfies $S^2 = \Phi * {\rm id} * \Phi^{-1}$. 
\item The modular automorphism of 
$\B(E)$ is the algebra automorphism $\sigma$ of $B(E)$ defined by $\sigma(u) = E^{-1}E^t u  E^{-1}E^t$, i.e.
$\sigma = \Phi* {\rm id} *\Phi$. 
\item We denote by $\theta$ is the algebra anti-automorphism of $\mathcal B(E)$ defined by
$\theta = S*\Phi * \Phi$, i.e. $\theta(u) = S(u)E^{-1}E^t
 E^{-1}E^t$. We have $S \circ \theta=\sigma$.
\end{enumerate}

\section{Free Yetter-Drinfeld modules}

In this section we introduce the concept of free Yetter-Drinfeld module, which will be essential for our purpose.
We begin by recalling the basics on Yetter-Drinfeld modules.

Let $A$ be a Hopf algebra. Recall that a (right-right) Yetter-Drinfeld
module over $A$ is a right $A$-comodule and right $A$-module $V$
satisfying the condition, $\forall v \in V$, $\forall a \in A$, 
$$(v \leftarrow a)_{(0)} \otimes  (v \leftarrow a)_{(1)} =
v_{(0)} \leftarrow a_{(2)} \otimes S(a_{(1)}) v_{(1)} a_{(3)}$$
The category of Yetter-Drinfeld modules over $A$ is denoted $\yd_A^A$:
the morphisms are the $A$-linear $A$-colinear maps.
Endowed with the usual tensor product of 
modules and comodules, it is a tensor category.

An important example of Yetter-Drinfeld module is the right coadjoint Yetter-Drinfeld
module $A_{\rm coad}$: as a right $A$-module $A_{\rm coad}=A$ and the right $A$-comodule structure
is defined by 
$${\rm ad}_r(a) = a_{(2)} \otimes S(a_{(1)})a_{(3)}, \forall a \in A$$
The following result, which will be of vital importance for us, generalizes
the construction of the right coadjoint comodule.

\begin{proposition}
 Let $A$ be  a Hopf algebra and let $V$ be a right $A$-comodule. 
 Endow $V \otimes A$ with the right $A$-module structure defined by multiplication on the right.
 Then the linear map
 \begin{align*}
  V \otimes A &\longrightarrow V \otimes A \otimes  A \\
  v \otimes a &\longmapsto v_{(0)}  \otimes a_{(2)} \otimes S(a_{(1)})v_{(1)}a_{(3)}
 \end{align*}
endows $V \otimes A$ with a right $A$-comodule structure, and with a Yetter-Drinfeld module structure.
We denote by $V \boxtimes A$ the resulting Yetter-Drinfeld module, and this constructions produces
a functor
 \begin{align*}
  L : \mathcal M^A \longrightarrow \yd_A^A \\
  V \longmapsto V \boxtimes A
 \end{align*}
\end{proposition}

\begin{proof}
 This is a direct verification.
\end{proof}

Note that when $V=\C$ is the trivial comodule, then $\C \boxtimes A= A_{\rm coad}$.

\begin{definition}
Let $A$ be a Hopf algebra.   A Yetter-Drinfeld module over $A$ is said to be free 
if it is isomorphic to $V \boxtimes A$ for some right $A$-comodule $V$.  
\end{definition}

Of course a free Yetter-Drinfeld is free as a right $A$-module.
The terminology is further justified by the following result.

\begin{proposition}\label{adjointfunctor}
 Let $A$ be a Hopf algebra. The functor
  $ L = - \boxtimes A: \mathcal M^A \longrightarrow \yd_A^A$ is left adjoint to
  the forgetful functor $ R :   \yd_A^A \longrightarrow \mathcal M^A$. 
 In particular if $P$ is a projective object in $\mathcal M^A$, then 
 $ L(P)$ is a projective object in $\yd_A^A$.   
\end{proposition}

\begin{proof}
Let $V \in \mathcal M^A$ and $X \in \mathcal \yd_A^A$. It is a direct verification to check
that we have a natural isomorphism
\begin{align*}
\Hom_{\mathcal M^A}(V, R(X))  &\longrightarrow  \Hom_{\yd_A^A}(V \boxtimes A, X)
  \\
 f &\longmapsto \tilde{f}, \ \tilde{f}(v\otimes a)=f(v)\leftarrow a
 \end{align*}
 and thus $L = -\boxtimes A$ is left adjoint the forgetful functor $R$.
The last assertion is a standard fact, see e.g. \cite{wei}, Proposition 2.3.10.
\end{proof}

It is worth to note that the existence of a left adjoint functor to  the forgetful functor $R :   \yd_A^A \longrightarrow \mathcal M^A$ follows from the general situations studied in \cite{camizh}.

Recall \cite{lin} that the category $\mathcal M^A$
of right $A$-comodules has enough projectives if and only if $A$ is co-Frobenius ($A$ is said to be co-Frobenius is there exists a non-zero right $A$-colinear map $A \rightarrow \C$).

\begin{corollary}\label{perfect}
 Let $A$ be a  co-Frobenius Hopf algebra.  Then the category $\yd_A^A$ has enough projective objects. %If moreover $A$ is cosemisimple, then the standard 
\end{corollary}

\begin{proof}
 Let $V \in \yd_A^A$ and let $P$ be  a projective object in $\mathcal M^A$ with an epimorphism $f : P\twoheadrightarrow R(V)$. We have a surjective morphism of Yetter-Drinfeld modules 
\begin{align*}
L(P) = P \boxtimes A &\twoheadrightarrow V \\
x \otimes a & \mapsto  (f(x)\leftarrow a)
\end{align*}  with $L(P)$ projective, and we are done.
\end{proof}

\begin{definition}\label{freeres}
 Let $A$ be a Hopf algebra and let $M \in \yd_A^A$. A free Yetter-Drinfeld  resolution 
 of $M$ consists
of a complex of free Yetter-Drinfeld modules 
$$\mathbf P_. = \cdots P_{n+1} \rightarrow P_n \rightarrow \cdots \rightarrow P_1 \rightarrow P_0\rightarrow 0$$ for which there exists a  Yetter-Drinfeld module
map $\epsilon : P_0 \rightarrow M$ such that 
$$\cdots P_{n+1} \rightarrow P_n \rightarrow \cdots \rightarrow P_1 \rightarrow P_0\overset{\epsilon}\rightarrow M \rightarrow 0$$
is an exact sequence.
\end{definition}

In particular a  free Yetter-Drinfeld resolution of $M$ is a resolution of  $M$ (as a right $A$-module)
by free $A$-modules.
A basic motivation for considering this special kind of resolutions comes from the fact that the standard
resolution of the counit is in fact a free Yetter-Drinfeld resolution. 
Before making this statement precise, we need the following construction.

For any $n \in \mathbb N$, we define the comodule
$A^{\boxtimes n}$ as follows:
$$A^0=\C, \ A^{\boxtimes 1} = \C \boxtimes A=A_{\rm coad}, \ A^{\boxtimes 2} = A^{\boxtimes 1} \boxtimes
A, \ \ldots, A^{\boxtimes (n+1)} = A^{\boxtimes n} \boxtimes A, \ldots$$ 
It is straighforward to check that after the obvious vector space identification of 
$A^{\boxtimes n}$ with $A^{\otimes n}$, the right $A$-module structure of $A^{\boxtimes n}$
is given by right multiplication and its comodule structure is given by
\begin{align*}
 {\rm ad}_r^{(n)} : A^{\boxtimes n}&\longrightarrow A^{\boxtimes n} \otimes A \\
 a_1 \otimes \cdots \otimes a_n &\longmapsto a_{1(2)} \otimes \cdots \otimes a_{n(2)}
 \otimes S(a_{1(1)} \cdots a_{n(1)}) a_{1(3)} \cdots a_{n(3)} 
\end{align*}

\begin{proposition}
 Let $A$ be Hopf algebra. The standard resolution of the counit of $A$ is a resolution of $\mathbb C$ by free Yetter-Drinfeld modules.
\end{proposition}

\begin{proof} It is a direct verification to check that for any $n \geq 0$, the map
 \begin{align*}
 A^{\boxtimes (n+1)} &\longrightarrow A^{\boxtimes n} \\
 a_1 \otimes \cdots \otimes a_{n+1} &\longmapsto \varepsilon(a_1) a_2 \otimes \cdots \otimes a_{n+1} + 
 \sum_{i=1}^n(-1)^i a_1 \otimes \cdots \otimes a_i a_{i+1} \otimes \cdots \otimes a_{n+1}
\end{align*}
is a morphism of Yetter-Drinfeld modules. This gives the result since the Yetter-Drinfeld modules $A^{\boxtimes (n+1)}$
are free by construction. 
\end{proof}

%Our interest for the functor $\mathcal L =-\boxtimes A$
%comes from the following considerations. 
%First let us define for any $n \in \mathbb N$ the comodule $A^{\boxtimes n}$ as follows:
%$$A^0=k, \ A^{\boxtimes 1} = k \boxtimes A=A_{\rm coad}, \ A^{\boxtimes 2} = A^{\boxtimes 1} \boxtimes
%A, \ \ldots, A^{\boxtimes (n+1)} = A^{\boxtimes n} \boxtimes A, \ldots$$ 

We close the section by recording the following elementary result, to be used in Section 5. The proof is left to the reader.

\begin{lemma}\label{Phi}
Let $A$ be a Hopf algebra,  let $V$ be finite-dimensional $A$-comodule with coaction $\alpha : V \rightarrow V \otimes A$, with basis $e_1, \ldots , e_n$, and let  
$(u_{ij}) \in M_n(A)$ be such that $\alpha(e_i)=\sum_{k}e_k \otimes u_{ki}$. The linear maps
 \begin{align*}
  \Phi_V^1 : V^*\otimes V &\longrightarrow \C \boxtimes A = A_{\rm coad} \\
  e_{i}^* \otimes e_j & \longmapsto u_{ij} 
 \end{align*}
\begin{align*}
  \Phi_V^2 : V^*\otimes V^* \otimes V \otimes V &\longrightarrow (V^* \otimes V)  \boxtimes A \\
  e_{i}^* \otimes e_j^* \otimes e_k \otimes e_l& \longmapsto e_j^* \otimes e_k \otimes u_{il} 
 \end{align*}
 are $A$-colinear.
\end{lemma}

\section{Equivalences between tensor categories of comodules}

In this section we present the technical core of the paper: the fact 
that  if $A$ and $H$ are Hopf algebras that have equivalent tensor categories of comodules,
then one can transport a free  Yetter-Drinfeld resolution of the counit of $A$ to 
the same kind of resolution for the counit of $H$ (with preservation of the length of the resolution).
The precise result is as follows.

\begin{theorem}\label{equivyetter}
 Let $A$ and $H$ be some Hopf algebras. Assume that there exists an equivalence of linear tensor categories
 $\Theta  : \mathcal M^A \simeq^{\otimes} \mathcal M^H$. Then $\Theta$ induces an equivalence of linear tensor categories
 $\widehat{\Theta} : \yd_{A}^A \simeq^\otimes \yd_H^H$ together with, for all $V \in \mathcal M^A$, natural isomorphisms
 $$\widehat{\Theta}(V \boxtimes A) \simeq \Theta(V) \boxtimes H$$
The functor $\widehat{\Theta}$ associates to any free Yetter-Drinfeld resolution of the counit of $A$
$$V_\cdot \boxtimes A :\cdots V_{n+1} \boxtimes A \rightarrow V_n \boxtimes A\rightarrow \cdots \rightarrow V_0\boxtimes A \rightarrow 0$$ a free Yetter-Drinfeld resolution of the counit of $H$
 $$\Theta(V_\cdot) \boxtimes H :\cdots \Theta(V_{n+1}) \boxtimes H \rightarrow \Theta(V_n) \boxtimes H\rightarrow \cdots \rightarrow\Theta(V_0)\boxtimes H \rightarrow 0$$
\end{theorem}

\begin{proof}
 Let $R_A : \yd_{A}^A \longrightarrow \mathcal M^A$ and $R_H : \yd_{H}^H \longrightarrow  \mathcal M^H$ be the respective forgetful functors with their respective left adjoint 
  $L_A : \mathcal M^A \longrightarrow \yd_A^A$ and  $L_H : \mathcal M^H \longrightarrow \yd_H^H$. The description of $\yd_A^A$ as the weak center of the monoidal category $\mathcal M_A$ (see e.g. \cite{scsurv}, \cite{kas})
ensures the existence of an equivalence of linear tensor categories $\widehat{\Theta} : \yd_{A}^A \simeq \yd_H^H$ such that $R_H \widehat{\Theta} \simeq \Theta R_A$ as functors. Denote by $\Theta^{-1}$ a quasi-inverse of $\Theta$. Then we have, for any $U \in \mathcal M^H$ and $X \in \yd_H^H$, natural isomorphisms
  \begin{align*}
   \Hom_{\yd_H^H}(\widehat{\Theta} L_A \Theta^{-1}(U), X) & \simeq \Hom_{\yd_A^A}(L_A \Theta^{-1}(U), \widehat{\Theta}^{-1}(X))\\
   & \simeq \Hom_{\mathcal M^A}(\Theta^{-1}(U), R_A\widehat{\Theta}^{-1}(X)) \\
   & \simeq \Hom_{\mathcal M^A}(\Theta^{-1}(U), \Theta^{-1}R_B(X)) \\
   & \simeq  \Hom_{\mathcal M^H}(U, \mathcal R_B(X))
  \end{align*}
The uniqueness of adjoint functors ensures that $\widehat{\Theta}  L_A \Theta^{-1} \simeq L_H$, so that $\widehat{\Theta} L_A \simeq  L_H \Theta$, as required. The last assertion is then immediate.
\end{proof}

In the next section, in order to transport an explicit resolution, we will need to
know the explicit form of a tensor equivalence $\Theta  : \mathcal M^A \simeq^{\otimes} \mathcal M^H$
and of the associated  tensor equivalence $\widehat{\Theta} : \yd_{A}^A \simeq \yd_H^H$.

It was shown by Schauenburg \cite{sc1} that  equivalences of linear tensor categories $\mathcal M^A \simeq^{\otimes} \mathcal M^H$ always arise from Hopf $A$-$H$-bi-Galois objects. 
The axioms of Hopf bi-Galois objects were symmetrized \cite{bi2,gru2}, leading to the use of the language of cogroupoids \cite{bicogro}, that  we now recall.  

First recall that a cocategory  $\coc$ consists of:

\noindent
$\bullet$ a set of objects ${\rm ob}(\coc)$.

\noindent
$\bullet$ For any $X,Y \in {\rm ob}(\coc)$, an algebra 
$\coc(X,Y)$. 

\noindent
$\bullet$ For any $X,Y,Z \in {\rm ob}(\coc)$, algebra morphisms
$$\Delta_{X,Y}^Z : \coc(X,Y) \longrightarrow \coc(X,Z) \otimes \coc(Z,Y)
\quad {\rm and} \quad \varepsilon_X : \coc(X,X) \longrightarrow \C$$
such that some natural coassociativity and counit diagrams (dual to the usual associativity and unit diagrams in a category) commute.

A cogroupoid  $\coc$ consists of a cocategory $\coc$ together
with, for any $X,Y \in {\rm ob}(\coc)$, linear maps
$$S_{X,Y} : \coc(X,Y) \longrightarrow \coc(Y,X)$$
such that natural diagrams (dual to the invertibility diagrams in a groupoid) commute.
A cogroupoid with a single object is precisely a Hopf algebra.
A cogroupoid is said to be connected if for any $X,Y \in {\rm ob}(\coc)$, the algebra
$\mathcal C(X,Y)$ is non-zero. 

The following theorem is the cogroupoid reformulation of Schauenburg's results in \cite{sc1}, 
see \cite{bicogro}.

\begin{theorem}\label{equivcomod}
Let $\coc$ be a connected cogroupoid. Then for any $X,Y \in {\rm ob}(\coc)$
we have  linear equivalences of tensor categories that are inverse of each other
\begin{align*}
\mathcal M^{\coc(X,X)} & \simeq^\otimes \mathcal M^{\coc(Y,Y)} 
\quad & \mathcal M^{\coc(Y,Y)} & \simeq^\otimes \mathcal M^{\coc(X,X)}\\
V &\longmapsto V \square_{\coc(X,X)} \coc(X,Y) \quad 
& V &\longmapsto V \square_{\coc(Y,Y)} \coc(Y,X)
\end{align*} 
Conversely, if $A$ and $H$ are Hopf algebras such that $\mathcal M^A \simeq^{\otimes} \mathcal M^H$,
then there exists a connected cogroupoid with 2 objects $X$, $Y$ such that $A = \mathcal C(X,X)$ and $B=\mathcal C(Y,Y)$. 
\end{theorem}

Here the symbol $\square$ stands for the cotensor product of a right comodule by a left comodule, see e.g. \cite{mon}. 

In order to extend the previous monoidal equivalences to categories of Yetter-Drinfeld modules, let us now recall Sweedler's notation for cocategories and cogroupoids.
Let $\coc$ be a cocategory. For $a^{XY} \in \coc(X,Y)$, we write 
$$\Delta_{X,Y}^Z(a^{XY})= a_{(1)}^{XZ} \otimes a_{(2)}^{ZY}$$
The cocategory axioms are
$$(\Delta_{X,Z}^T \otimes 1)\circ \Delta_{X,Y}^Z(a^{XY}) = 
a_{(1)}^{XT} \otimes a_{(2)}^{TZ} \otimes a_{(3)}^{ZY}= 
(1 \otimes \Delta_{T,Y}^Z)\circ \Delta_{X,Y}^T(a^{XY})
$$
$$\varepsilon_X(a_{(1)}^{XX}) a_{(2)}^{XY} = a^{XY} =
\varepsilon_Y(a_{(2)}^{YY})  a_{(1)}^{XY}$$
and the additional cogroupoid axioms are
$$S_{X,Y}(a_{(1)}^{XY}) a_{(2)}^{YX} = \varepsilon_X(a^{XX})1 =
a_{(1)}^{XY} S_{Y,X}(a_{(2)}^{YX})
$$
The following result is Proposition 6.2 in \cite{bicogro}. 

\begin{proposition}
 Let $\coc$ be cogroupoid, let $X,Y \in {\rm ob}(\coc)$ and let $V$
be a right $\coc(X,X)$-module.
\begin{enumerate}
 \item $V \otimes \coc(X,Y)$ has a right $\coc(Y,Y)$-module structure
defined by
$$\left(v \otimes a^{XY}\right) \leftarrow b^{YY} = v\leftarrow b_{(2)}^{XX}
\otimes S_{YX}(b_{(1)}^{YX})a^{XY}b_{(3)}^{XY}$$ 
Endowed with the right $\coc(Y,Y)$-comodule defined by $1 \otimes \Delta_{X,Y}^Y$,
$V \otimes \coc(X,Y)$ is a Yetter-Drinfeld module over $\coc(Y,Y)$.
\item  If moreover $V$ is Yetter-Drinfeld module, then $V \square_{\coc(X,X)} \coc(X,Y)$ is Yetter-Drinfeld submodule
of $V \otimes \coc(X,Y)$.  
\end{enumerate}
\end{proposition}

We now can write down the explicit form of the tensor equivalence between categories of Yetter-Drinfeld modules
induced by a tensor equivalence between categories of comodules.  

\begin{theorem}\label{expliyetter}
 Let $\coc$ be connected cogroupoid. Then for any $X,Y \in {\rm ob}(\coc)$,
the functor
\begin{align*}
 \yd_{\coc(X,X)}^{\coc(X,X)} & \longrightarrow \yd_{\coc(Y,Y)}^{\coc(Y,Y)} \\
V &\longmapsto V \square_{\coc(X,X)} \coc(X,Y)
\end{align*}
is an equivalence of linear tensor categories. Moreover we have natural isomorphisms
\begin{align*}
 \left(V \square_{\coc(X,X)}\coc(X,Y)\right) \boxtimes \coc(Y,Y) 
 & \longrightarrow \left(V \boxtimes \coc(X,X)\right)\square_{\coc(X,X)} \coc(X,Y)\\
 v \otimes a^{XY} \otimes b^{YY} &\longmapsto v \otimes b_{(2)}^{XX} \otimes S_{Y,X}(b_{(1)}^{YX})a^{XY}b_{(3)}^{XY}
\end{align*}
\end{theorem}

\begin{proof}
 The fact that this indeed defines an equivalence of tensor categories is proved in \cite{bicogro}, Theorem 6.3.
The announced natural isomorphism is the one induced by the uniqueness of adjoint functors as in the proof of Theorem \ref{equivyetter}. For the reader's convenience, let us write down explicitely the inverse isomorphism.
Let 
\begin{align*}
 \xi_V : & \left(V \boxtimes \coc(X,X)\right)\square_{\coc(X,X)} \coc(X,Y)  \longrightarrow \\
& \left(\left(V \square_{\coc(X,X)}\coc(X,Y)\right) \boxtimes \coc(Y,Y)\right) \square_{\mathcal C(Y,Y)} \left(\mathcal C(Y,X)\square_{\coc(X,X)}\coc(X,Y)\right) \\
v & \otimes a^{XX} \otimes b^{XY}  \longmapsto
v_{(0)} \otimes v_{(1)}^{XY} \otimes a_{(2)}^{YY} \otimes S_{XY}(a_{(1)}^{XY})v_{(2)}^{YX}a_{(3)}^{YX} \otimes b^{XY} 
\end{align*}
The explicit inverse of the morphism in the statement is then $({\rm id} \otimes (\varepsilon_Y  f))\xi_V$, where $f$ is the inverse isomorphism to $\Delta_{YY}^X : \coc(Y,Y) \rightarrow \coc(Y,X) \square_{\coc(X,X)} \coc(X,Y)$, see the proof of Lemma 2.14 in \cite{bicogro}.
\end{proof}

We end the section by recalling that $\B(E)$ is  part of a cogroupoid.
Let $E \in \GL_m(\C)$ and let $F \in \GL_n(\C)$.
Recall \cite{bi1} that the algebra $\mathcal B(E,F)$ is the universal algebra with generators
$u_{ij}$, $1\leq i \leq m,1\leq j \leq n$, 
satisfying the relations
$$F^{-\!1} u^t E u = I_n \ ; \   u F^{-\!1}  u^t E = I_m.$$
Of course the generator $u_{ij}$ in $\B(E,F)$ 
is denoted $u_{ij}^{EF}$ to express the dependence on $E$ and $F$, when needed.
 It is clear that $\B(E,E)=\B(E)$. 

We get a cogroupoid $\mathcal B$ whose objects are the invertible matrices $E \in \GL_n(\C)$, where the algebras 
$\B(E,F)$ are the ones just defined and where the structural morphisms are the algebra maps defined as follows
\begin{align*}
\Delta_{E,F}^G : \B(E,F) &\longrightarrow \B(E,G) \otimes \B(G,F)\\
u_{ij}^{EF} &\longmapsto \sum_k u_{ik}^{EG} \otimes u_{kj}^{GF}
\end{align*} 
\begin{align*}
S_{E,F} :\B(E,F) & \longrightarrow \B(F,E)^{\rm op} \\
u & \longmapsto E^{-1}u^tF
\end{align*}
and where $\varepsilon_{E}$ is the counit of $ \B(E)$.

When $\B(E,F)\not=0$ (i.e. when ${\rm tr}(E^{-1}E^t)={\rm tr}(F^{-1}F^t)$ and the matrices $E$, $F$ have size $\geq 2$, see \cite{bi1}), we know, by Theorem \ref{equivcomod} and Theorem \ref{expliyetter}, that the cotensor product by $\B(E,F)$ induces equivalences of tensor categories
$$\mathcal M^{\B(E)} \simeq^\otimes \mathcal M^{\B(F)}, \quad \yd^{\B(E)}_{\B(E)} \simeq^\otimes \yd^{\B(F)}_{\B(F)}$$ 
The  following $\B(F)$-comodule isomorphisms will be used in the next section.
\begin{align*}
 V_F & \longrightarrow V_E \square_{\B(E)} \B(E,F), \ & V_F^* \otimes V_F &\longrightarrow (V_E^*\otimes V_E)\square_{\B(E)}\B(E,F) \\
e_i^F & \longmapsto \sum_k e_k\otimes u_{ki}^{EF}, \ 
& e_i^{F*} \otimes e_j^F & \longmapsto \sum_{k,l} e_k^{E*} \otimes e_l^E \otimes S_{FE}(u_{ik}^{FE})u_{lj}^{EF}
\end{align*}

\section{A resolution of the counit for $\mathcal B(E)$}

In this section we write down the announced resolution $\star_E$ for the counit of $\B(E)$.

\subsection{The resolution.}
\begin{theorem}\label{reso}
 Let $E \in \GL_n(\C)$, $n \geq 2$, and let $V_E$ be the fundamental $n$-dimensional $\mathcal B(E)$-comodule.
 There exists an exact sequence of Yetter-Drinfeld modules over $\B(E)$
 $$0 \to \C\boxtimes \mathcal B(E) \overset{\phi_1}\longrightarrow
  (V_E^* \otimes V_E) \boxtimes \B(E) \overset{\phi_2}\longrightarrow
 (V_E^* \otimes V_E) \boxtimes \B(E) \overset{\phi_3} \longrightarrow \C \boxtimes \B(E) \overset{\varepsilon} \longrightarrow \C \to 0$$
 which thus yields a  free Yetter-Drinfeld resolution of the counit of $\B(E)$.
\end{theorem}

Of course the first thing to do is to define the maps $\phi_1$, $\phi_2$, $\phi_3$ in Theorem
\ref{reso}.

\begin{definition}
 Let $e_1^E, \ldots, e_n^E$ be the canonical basis of $V_E$.
 The linear maps   $\phi_1$, $\phi_2$, $\phi_3$ in Theorem \ref{reso}
 are defined as follows.
 \begin{align*}
  \phi_1 : \C\boxtimes \mathcal B(E) &\longrightarrow(V_E^*\otimes V_E)\boxtimes \mathcal B(E) \\
  x &\longmapsto \sum_{i,j} e_i^{E*}\otimes e^E_j \otimes
  \left( (E^tE^{-1})_{ij} - (Eu(E^{t})^{-1})_{ij}\right)x
 \end{align*}
 \begin{align*}
\phi_2 : (V_E^*\otimes V_E)\boxtimes \mathcal B(E) &\longrightarrow (V_E^*\otimes V_E)\boxtimes \mathcal B(E) \\
 e_i^{E*} \otimes e_j^E \otimes x &\longmapsto e_i^{E*} \otimes e_j^E \otimes x + 
 \sum_{k,l}e_k^{E*} \otimes e_l^E \otimes (u(E^t)^{-1})_{il}E_{jk}x
 \end{align*}
 \begin{align*}
  \phi_3 : (V_E^*\otimes V_E)\boxtimes \mathcal B(E) &\longrightarrow \mathcal \C \boxtimes B(E) \\
  e_i^{E*} \otimes e_j^{E} \otimes x &\longmapsto (u_{ij}-\delta_{ij})x
 \end{align*}
\end{definition}

When $E=I_n$, the maps $\phi_1$, $\phi_2$, $\phi_3$ are those defined in \cite{cht}.

\subsection{Proof of Theorem \ref{reso}}
\begin{lemma}
 The maps  $\phi_1$, $\phi_2$, $\phi_3$ in Theorem \ref{reso} are morphisms of Yetter-Drinfeld modules.
\end{lemma}

\begin{proof}
 This can be checked directly, but for future use we describe $\phi_1$, $\phi_2$, $\phi_3$
 as linear combination of maps that are known to be morphisms of Yetter-Drinfeld modules.
 Let $\phi'_1$, $\phi_1''$ be defined by the following compositions of $\mathcal B(E)$-colinear maps
 $$\phi_1' : \C  \overset{\delta}  \rightarrow V_E \otimes V_E\overset{{\varphi\otimes {\rm id}}} \longrightarrow V_E^* \otimes V_E \overset{{\rm id} \otimes u}
 \longrightarrow (V_E^* \otimes V_E) \boxtimes \mathcal B(E)$$
 $$\phi_1'' : \C  \overset{\delta\otimes \delta}  \longrightarrow V_E \otimes V_E\otimes V_E \otimes V_E \overset{{\varphi\otimes \varphi \otimes {\rm id}}} \longrightarrow V_E^* \otimes V_E^* \otimes V_E\otimes V_E\overset{\Phi^2_V}\longrightarrow (V_E^* \otimes V_E) \boxtimes \mathcal B(E) $$
where $u$ is the unit map $k \subset \mathcal B(E)$, the maps $\varphi$, $\delta$ were defined in Section 2 and $\Phi_V^2$ was defined in Lemma \ref{Phi}. We have $\phi_1 = \tilde{\phi_1'} - \tilde{\phi_1''}$ (where the notation $\tilde{}$ has the same meaning as in the proof of Proposition \ref{adjointfunctor}) 
hence $\phi_1$ is a Yetter-Drinfeld map. Define now $\phi_2'$ by the composition of the following colinear maps
$$\phi_2' :  V_E^* \otimes V_E  \overset{{\rm id}\otimes \delta}  \longrightarrow 
V_E^* \otimes V_E\otimes V_E \otimes V_E \overset{{\rm id}\otimes \varphi \otimes {\rm id}} \longrightarrow V_E^* \otimes V_E^* \otimes V_E\otimes V_E\overset{\Phi^2_V}\longrightarrow (V_E^* \otimes V_E) \boxtimes \mathcal B(E) $$
We have $\phi_2 = {\rm id} + \tilde{\phi_2'}$, hence 
$\phi_2$ is a Yetter-Drinfeld map. Finally we have
$\phi_3 = \tilde{\Phi^1_V} - \tilde{\phi_3'}$, where $\phi_3'$ is the evaluation map 
$V^*\otimes V \rightarrow k \subset \mathcal B(E)$, and hence $\phi_3$ is also a morphism of 
Yetter-Drinfeld modules. 
\end{proof}

\begin{lemma}
 The sequence in Theorem \ref{reso} is a complex.
\end{lemma}

\begin{proof}
 It is straighforword to check that $\varepsilon \circ \phi_3=0$ and that $\phi_3 \circ \phi_2=0$.
The identity $\phi_2 \circ \phi_1=0$ follows from the observation that
$\tilde{\phi_2'} \circ \tilde{\phi_1'} = \tilde{\phi_1''}$ and
$\tilde{\phi_2'} \circ \tilde{\phi_1''} = \tilde{\phi_1'}$.
\end{proof}

\begin{lemma}\label{compar}
 Let $E \in \GL_n(\C)$, $F \in \GL_m(\C)$ with $m,n \geq 2$ and ${\rm tr}(E^{-1}E^t)={\rm tr}(F^{-1}F^t)$.
Then the sequence in Theorem \ref{reso} is exact for $\mathcal B(E)$ if and only if it is exact for $\mathcal B(F)$.
\end{lemma}

\begin{proof}
 We know from \cite{bi1} and Theorems \ref{equivyetter} and \ref{expliyetter}
that we have an equivalence of tensor categories $\yd_{\mathcal B(E)}^{\mathcal B(E)} \simeq^\otimes  \yd_{\mathcal B(F)}^{\mathcal B(F)}$ that preserves freeness of Yetter-Drinfeld modules. Let us check that this tensor equivalence transforms the complex of Yetter-Drinfeld modules
of Theorem \ref{reso} for $\mathcal B(E)$ into the complex of Yetter-Drinfeld modules of Theorem \ref{reso} for $\mathcal B(F)$.
First consider the following diagram.
$$\xymatrix{
\C \boxtimes \mathcal B(F)  \ar[r]^{\phi_1^F} \ar[d] & (V_F^* \otimes V_F) \boxtimes \B(F) \ar[d]  \\
\left(\C \square_{\mathcal B(E)} \mathcal B(E,F)\right)\boxtimes \mathcal B(F) \ar[d] &  
\left((V_E^* \otimes V_E) \square_{\mathcal B(E)} \mathcal B(E,F)\right)\boxtimes \B(F) \ar[d]\\
\left(\C \boxtimes \mathcal B(E)\right) \square_{\mathcal B(E)}\mathcal B(E,F)\ar[r]^(.45){\phi_1^E\otimes {\rm id}} &  
\left((V_E^* \otimes V_E) \boxtimes \mathcal B(E)\right) \square_{\mathcal B(E)} \mathcal B(E,F)
}$$

The first vertical arrow on the left is given by the identification $\C \simeq \C\square_{\B(E)} \B(E,F)$
and the second one is that of Theorem \ref{expliyetter} for the trivial comodule $\C$.
 So the composition of the vertical arrows on the left is
\begin{align*}
\C \boxtimes \mathcal B(F) & \longrightarrow
\left(\C \boxtimes \mathcal B(E)\right) \square_{\mathcal B(E)}\mathcal B(E,F) \\
 1 \otimes x^{FF}&\longmapsto x_{(2)}^{EE} \otimes S_{FE}(x_{(1)}^{FE})x_{(3)}^{EF}
 \end{align*}
The first vertical arrow on the right is the one given at the end of Section 4, while the second one is that of Theorem \ref{expliyetter} for the comodule $V_E^* \otimes V_E$.
So the composition of the vertical arrows on the right is
\begin{align*}
(V_F^* \otimes V_F) \boxtimes \B(F) & \longrightarrow
\left((V_E^* \otimes V_E) \boxtimes \B(E) \right)\square_{\mathcal B(E)} \mathcal B(E,F)\\
 e_i^{F*} \otimes e_j^F \otimes x^{FF} & \longmapsto
\sum_{k,l} e_k^{E*} \otimes e_l^E \otimes x_{(2)}^{EE} \otimes S_{FE}(x_{(1)}^{FE})
S_{FE}(u_{ik}^{FE})u_{lj}^{EF}x_{(3)}^{EF}
\end{align*}
The vertical arrows are compositions of isomorphisms so are isomorphims.
It is a direct verification to check that the previous diagram is commutative.
Similarly one checks the commutativity of the following diagrams.

$$\xymatrix{
(V_F^* \otimes V_F) \boxtimes \B(F)  \ar[r]^{\phi_2^F} \ar[d] &  \ar[d] (V_F^* \otimes V_F) \boxtimes \B(F) \\
\left((V_E^* \otimes V_E) \square_{\mathcal B(E)} \mathcal B(E,F)\right)\boxtimes \B(F) \ar[d] &  
\left((V_E^* \otimes V_E) \square_{\mathcal B(E)} \mathcal B(E,F)\right)\boxtimes \B(F) \ar[d]\\
\left((V_E^* \otimes V_E) \boxtimes \mathcal B(E)\right) \square_{\mathcal B(E)} \mathcal B(E,F)\ar[r]^{\phi_2^E\otimes {\rm id}} &  
\left((V_E^* \otimes V_E) \boxtimes \mathcal B(E)\right) \square_{\mathcal B(E)} \mathcal B(E,F)
}$$

$$\xymatrix{
(V_F^* \otimes V_F) \boxtimes \B(F)  \ar[r]^{\phi_3^F} \ar[d] &  \ar[d] \C \boxtimes \B(F) \\
\left((V_E^* \otimes V_E) \square_{\mathcal B(E)} \mathcal B(E,F)\right)\boxtimes \B(F) \ar[d] &  
\left(\C \square_{\mathcal B(E)} \mathcal B(E,F)\right)\boxtimes \B(F) \ar[d]\\
\left((V_E^* \otimes V_E) \boxtimes \mathcal B(E)\right) \square_{\mathcal B(E)} \mathcal B(E,F)\ar[r]^(.55){\phi_3^E\otimes {\rm id}} &  
\left(\C \boxtimes \mathcal B(E)\right) \square_{\mathcal B(E)} \mathcal B(E,F)
}$$

$$\xymatrix{
\C \boxtimes \mathcal B(F)  \ar[r]^{\varepsilon} \ar[d] & \C \ar[dd]  \\
\left(\C \square_{\mathcal B(E)} \mathcal B(E,F)\right)\boxtimes \mathcal B(F) \ar[d] &  \\
\left(\C \boxtimes \mathcal B(E)\right) \square_{\mathcal B(E)}\mathcal B(E,F)\ar[r]^(.6){\varepsilon\otimes {\rm id}} &  \C \square_{\mathcal B(E)} \mathcal B(E,F)
}$$
The vertical isomorphims are those previously defined.
Thus we conclude that the complex in Theorem \ref{reso} is exact for $\mathcal B(E)$ if and only if it is exact for $\mathcal B(F)$.
\end{proof}

\begin{lemma}\label{eq}
 Let $q \in \C^*$. The sequence of Theorem \ref{reso} is exact when $E = E_q$.
\end{lemma}

\begin{proof}
Put $A= \B(E_q)=\mathcal O(\SL_q(2))$. As usual we put $a=u_{11}$, $b=u_{12}$, $c=u_{21}$, $d=u_{22}$.
We will frequently use the well-known fact that $A$ and its quotients $A/(b)$, $A/(c)$ and $A/(b,c)$ are integral domains.
For $x \in A$, we have
 \begin{align*}
 \phi_1(x) = e_1^* \otimes e_1 \otimes &((-q^{-1} +qd)x )+
 e_1^* \otimes e_2 \otimes (-cx) + e_2^* \otimes e_1 \otimes (-bx) +
 e_2^* \otimes e_2 \otimes ((-q+q^{-1}a)x) \\
  \phi_2(e_1^* \otimes e_1 \otimes x) &= e_1^* \otimes e_1 \otimes x
  + e_2^* \otimes e_1 \otimes (-qbx) + e_2^* \otimes e_2 \otimes ax \\
\phi_2 (e_1^* \otimes e_2 \otimes x) &=
e_1^* \otimes e_1 \otimes bx + e_1^* \otimes e_2 \otimes (1-q^{-1}a)x \\
\phi_2(e_2^* \otimes e_1 \otimes x) &= e_2^* \otimes e_1 \otimes (1-qd)x
+e_2^* \otimes e_2 \otimes cx \\
\phi_2(e_2^* \otimes e_2 \otimes x) &= e_1^* \otimes e_1 \otimes dx + e_1^* \otimes e_2 \otimes (-q^{-1}cx)
+ e_2^* \otimes e_2 \otimes x \\
\phi_3(e_1^* \otimes e_1 \otimes x) &= (a-1)x, \quad \phi_3(e_1^* \otimes e_2 \otimes x) = bx, \\
\phi_3(e_2^* \otimes e_1 \otimes x) &=cx, \quad \phi_3(e_2^* \otimes e_2 \otimes x) =(d-1)x
 \end{align*}
The injectivity of $\phi_1$ follows from the fact that $A$ is an integral domain and the surjectivity of 
$\phi_3$ is easy (and well-known). Let $X = \sum_{i,j}e_i^* \otimes e_j \otimes x_{ij} \in \Ker(\phi_3)$.
We have
$$X + \phi_2(-e_1^* \otimes e_1 \otimes x_{11}) = e_1^* \otimes e_2\otimes x_{12} + e_2^*\otimes e_1
\otimes (qbx_{11} + x_{21}) + e_2^* \otimes e_2 \otimes (-ax_{11} +x_{22})$$ and
hence to show that $X \in {\rm Im}(\phi_2)$, we can assume that $x_{11}=0$. We have
$$bx_{12} + cx_{21} + (d-1)x_{22}=0$$
which gives $(d-1)x=0$ in the integral domain $A/(b,c)$ and thus $x_{22} = b\alpha +c\beta$ for some 
$\alpha, \beta \in A$. Then we have
$$X+ \phi_2 ( e_1^* \otimes e_2 \otimes qd\alpha - e_2^* \otimes e_1 \otimes \beta - e_2^* \otimes e_2 \otimes b\alpha) = e_1^* \otimes e_2 \otimes x + e_2^* \otimes e_1 \otimes y$$
for some $x,y \in A$, and hence we also can assume that $x_{22}=0$.  
Then we have $bx_{12}+cx_{21}=0$, which gives $bx_{12}=0$ in the integral domain 
$A/(c)$, and hence $x_{12}=c\alpha$ for some $\alpha \in A$, and moreover $x_{21}=-b\alpha$.
Then we have
$$\phi_2(q^{-1} e_1^* \otimes e_1 \otimes \alpha + e_1^* \otimes e_2 \otimes c\alpha 
-q^{-1} e_2^* \otimes e_2 \otimes a\alpha) = X$$
and we conclude that $\Ker(\phi_3) = {\rm Im}(\phi_2)$.

Let $X = \sum_{i,j}e_i^* \otimes e_j \otimes x_{ij} \in \Ker(\phi_2)$.
Then $-qbx_{11} +(1-qd)x_{21}=0$, hence $(1-qd)x_{21}=0$ in the integral domain
$A/(b)$ and hence $x_{21}= b\alpha$ for some $\alpha \in A$.
Hence 
$$X + \phi_1(\alpha) = e_1^* \otimes e_1 \otimes (x_{11} + (-q^{-1}+qd)\alpha)
+ e_1^* \otimes e_2 \otimes (x_{12}-c\alpha) + e_2^* \otimes e_2 \otimes (x_{22} + (-q+q^{-1}a)\alpha)$$
and we can assume that $x_{21}=0$. But then, using the fact that $A$ is an integral domain,
we see that $X=0$ since $X \in \Ker(\phi_2)$. We conclude that $\Ker(\phi_2)= {\rm Im}(\phi_1)$.
\end{proof}

We are now ready to prove Theorem \ref{reso}. Let $E \in \GL_n(\C)$, $n \geq 2$, and let $q \in \C^*$ be such that
${\rm tr}(E^{-1}E^t)=-q-q^{-1}={\rm tr}(E_q^{-1}E_q^t)$. Lemma \ref{compar} and Lemma \ref{eq} combined together yield that the sequence in Theorem \ref{reso} is exact. 

\section{Applications}

%The first application of Theorem \ref{reso} is the following result.

In this section we present several applications of the resolution built in the previous section.

\subsection{Smoothness and Poincar\'e duality}

\begin{theorem}
 Let $E \in \GL_n(\C)$ with $n \geq 2$. The algebra $\B(E)$ is smooth of dimension 3. In particular
 $H_n(\B(E), M)=(0)=H^n(\B(E),M)$ for any $\B(E)$-bimodule $M$ and any $n \geq 4$.
\end{theorem}

\begin{proof} Put $A = \B(E)$, and consider  the algebra map 
$\Delta' : A \rightarrow A^e = A \otimes A^{\rm op}$, $a \mapsto a_{(1)} \otimes S^{-1}(a_{(2)})$. It induces an exact functor
$\mathcal M_A \longrightarrow \mathcal M_{A^e}$, $M \longmapsto M\otimes_A A^e$ that sends free $A$-modules to free $A^e$-modules and the trivial module $\C_\varepsilon$ to the trivial bimodule $A$ (since $A^e$ is free for the left $A$-module structure induced by $\Delta'$, see e.g. Subsection 2.2 in \cite{bz}). Thus $A$ has a length 3 resolution  by finitely generated free $A^e$-modules and is smooth of dimension $d\leq 3$.  
Moreover using Proposition \ref{ident} and Theorem \ref{reso},  we see that $H_3(\B(E), _\Phi\!k_{\Phi^{-1}}) \simeq \C$ (see more generally the next subsection), so the  resolution in Theorem \ref{reso} has minimal length and we conclude that $A$ is smooth of dimension 3.
\end{proof}

We now show that Poincar\'e duality holds for the algebras $\B(E)$.

\begin{proposition}\label{poincare}
 Let $M$ be a right $\B(E)$-module. Then for any $n\in \{0,1,2 , 3\}$ we have isomorphisms
 $${\rm Ext}_{\B(E)}^n(\C_\varepsilon, M) \simeq {\rm Tor}_{3-n}^{\B(E)}(\C_\varepsilon, \ _{\theta}\!M)$$
 where $\theta$ is the algebra anti-automorphism of $\mathcal B(E)$ defined by $\theta(u) = S(u)E^{-1}E^t
 E^{-1}E^t$ and where $_\theta\!M$ has the left $\B(E)$-module structure given by $a\cdot x := x\cdot \theta(a)$.
\end{proposition}

\begin{proof}
 After applying the functor $\Hom_A(-,M)$ to the resolution of Theorem \ref{reso} and standard identifications, the complex to compute the Ext-groups on the the left becomes
 $$0 \to M \overset{\phi_3^t}\longrightarrow
  V^* \otimes V \otimes M \overset{\phi_2^t}\longrightarrow
 V^* \otimes V \otimes M \overset{\phi_1^t} \longrightarrow M   \to 0$$
 where $V=V_E$ and with 
 $$\phi_3^t(x) = \sum_{i,j}e_i^*\otimes e_j \otimes x \cdot(u_{ji}-\delta_{ji})$$
 $$\phi_2^t(e_i^*\otimes e_j \otimes x) = e_i^*\otimes e_j \otimes x +
 \sum_{k,l} e_k^*\otimes e_l \otimes x\cdot \left((u(E^t)^{-1})_{li}E_{kj}\right)$$
 $$\phi_1^t(e_i^*\otimes e_j \otimes x) = x\cdot\left((E^tE^{-1})_{ji}-(Eu(E^t)^{-1})_{ji}\right)$$
 Consider now the isomorphisms
 $\iota : M \to A \otimes_A (_{\theta}\!M)$, $x \mapsto 1 \otimes_A x$, and 
 $$f : V^* \otimes V \otimes M \to V^* \otimes V \otimes A \otimes_A (_{\theta}\!M), \ \phi \otimes v \otimes x
  \mapsto \phi \otimes f_0(v) \otimes 1  \otimes_A x$$
  where $f_0$ is the automorphism of $V$ whose matrix in the canonical basis is $(E^t)^{-1}E$.
  The following diagrams commute
{\small
$$\xymatrix{
0 \ar[r] & M \ar[r]^{\phi_3^t} \ar[d]^{\iota}&
  V^* \otimes V \otimes M\ar[r]^{\phi_2^t} \ar[d]^f &
 V^* \otimes V \otimes M \ar[r]^{\phi_1^t}  \ar[d]^f &
  M  \ar[d]^{\iota} \ar[r] & 0 \\
0 \ar[r] & A\otimes_A (_{\theta}\!M) \ar[r]^(0.4){\phi_1\otimes_A {\rm id} } &
  V^* \otimes V \otimes A\otimes_A(_{\theta}\!M) \ar[r]^{\phi_2\otimes_A {\rm id}}&
 V^* \otimes V \otimes A \otimes_A (_{\theta}\!M) \ar[r]^(.6){ \phi_3\otimes_A {\rm id}} &  A  \otimes_A(_{\theta}\!M) \ar[r] & 0 \\
}$$}
and hence since the homology of the lower complex gives the Tor-groups in the proposition, we get
the result.
\end{proof}

\begin{corollary}\label{poinhoch}
 Let $M$ be a $\B(E)$-bimodule.
 Then for any $n\in \{0 ,1, 2, 3 \}$ we have isomorphisms
 $$H^n(\B(E), M) \simeq H_{3-n}(\B(E), \ _{\sigma}\!M)$$
 where $\sigma$ is the modular automorphism of $\B(E)$ given by $\sigma(u) = E^{-1}E^t u
 E^{-1}E^t$.
\end{corollary}

\begin{proof}
 We know for Proposition \ref{ident} that $$H_{3-n}(\B(E), \ _{\sigma}\!M) \simeq 
 {\rm Tor}_{3-n}^{\B(E)}(\C_\varepsilon,
 (_{\sigma}\!M)')$$
  where the left $\B(E)$-module structure on $(_{\sigma}\!M)'$ is given by
 $a\to x=\sigma(a_{(2)})\cdot x\cdot S(a_{(1)})$.
 On the other hand we have by  Proposition \ref{ident} and Proposition \ref{poincare},
 $$H^n(\B(E),M) \simeq {\rm Ext}^n_{\B(E)}(\C_\varepsilon, M'') \simeq 
 {\rm Tor}_{3-n}^{\B(E)}(\C_\varepsilon, \ _{\theta}\!(M''))$$
 The left $\B(E)$-module structure on $_{\theta}\!(M'')$ is given by
 \begin{align*}
  a \rightsquigarrow x &= x \leftarrow \theta(a) = S(\theta(a)_{(1)}) \cdot x \cdot \theta(a)_{(2)} \\
  & =S\theta(a_{(2)})\cdot x \cdot S(a_{(1)}) = \sigma(a_{(2)}) \cdot x \cdot S(a_{(1)})\\
  & = a\rightarrow x
 \end{align*}
We conclude that $_{\theta}\!(M'') =(_{\sigma}\!M)'$ and we have our result.
\end{proof}

See \cite{bz} for more examples of (noetherian) Hopf algebras satisfying Poincar\'e duality.

\subsection{Some homology computations}
In this short subsection we record the computation of the Hochschild homology of $\B(E)$ when the bimodule of coefficients  has dimension 1 as a vector space.

\begin{proposition}
 Let $\alpha, \beta \in \Hom_{\C-{\rm alg}}(\B(E),\C)$. Put $\gamma = \beta^{-1}*\alpha$. Then

 \begin{equation*}
  H_0(\B(E), \ _\alpha \!\C_\beta) \simeq \begin{cases}
                         0 & \text{if $\alpha \not=\beta$} \\
			 \C & \text{if $\alpha=\beta$}
                        \end{cases} \quad 
	 H_3(\B(E), \ _\alpha \!\C_\beta) \simeq \begin{cases}
                         0 & \text{if $\alpha \not=\beta*\Phi^2$} \\
			 \C & \text{if $\alpha=\beta * \Phi^2$}
                        \end{cases}		
 \end{equation*}
$$H_1(\B(E),\  _\alpha \!\C_\beta) \simeq
\frac{\{M \in M_n(\C) \ | \ {\rm tr}(M\gamma(u)^t)={\rm tr}(M)\}}{\{M + E^tM^t\gamma(u)(E^t)^{-1}, \ M \in M_n(k)\}}$$
$$H_2(\B(E), \  _\alpha \!\C_\beta) \simeq 
\frac{\{M \in M_n(\C), \ M + E^tM^t\gamma(u)(E^t)^{-1}=0\}}{\{\lambda(\sum_{i,j} e_i^*\otimes e_j \otimes
  \left( (E^tE^{-1})_{ij} - (E\gamma(u)(E^{t})^{-1})_{ij}\right)), \lambda \in \C\}}$$
\end{proposition}

The proof is a direct computation by using Proposition \ref{ident} and Theorem \ref{reso}. The computation of the cohomology groups follows by using Poincar\'e duality. 

\subsection{Bialgebra cohomology of $\mathcal B(E)$}
The cohomology of a bialgebra was introduced by Gerstenhaber and 
Schack \cite{gs2,gs1}: it is defined by means of an explicit bicomplex whose
arrows are modelled on the Hochschild complex of the underlying algebra
and columns are modelled on the Cartier complex of the underlying coalgebra.
If $A$ is a Hopf algebra, let us denote by $H^*_{b}(A)$
the resulting cohomology. 
Taillefer \cite{tai04} proved that
$$H_b^*(A) \cong {\rm Ext}^*_{\mathcal M(A)}(A,A)$$
where $\mathcal M(A)$ is the category of Hopf bimodules over $A$.
Combined with the monoidal equivalence between Hopf bimodules and Yetter-Drinfeld modules \cite{sc94}, this yields an isomorphism 
$$H_b^*(A) \cong {\rm Ext}^*_{\yd_A^A}(\C,\C)$$
Bialgebra cohomology of algebras of polynomial functions on linear algebraic groups was studied by Parshall and Wang \cite{pw},  with a complete description in the connected reductive case.
It seems that only very few full computations of $H_b^*(A)$ are known in the non-commutative non-cocommutative case, see \cite{tai07} and the references therein.
 The following result gives in particular the description of the bialgebra cohomology of $\mathcal O(\SL_q(2))$ for $q$ generic, i.e. $q=\pm 1$ or $q$ not a root of unity.

\begin{theorem}
 Assume that $\mathcal B(E)$ is cosemisimple, i.e. that the solutions of the equation ${\rm tr}(E^{-1}E^t)=-q-q^{-1}$ are generic. Then we have
 \begin{equation*}
  H^n_{b}(\B(E)) \simeq \begin{cases}
                         0 & \text{if $n \not=0,3$} \\
			 \C & \text{if $n=0,3$}
                        \end{cases}
 \end{equation*}
\end{theorem}

\begin{proof}
 The assumption that $\B(E)$ is cosemisimple ensures that evey object in $\mathcal M^{\B(E)}$ is projective, so by Proposition \ref{adjointfunctor} every free Yetter-Drinfeld is a projective object in $\yd_{\B(E)}^{\B(E)}$ and the category $\yd_{\B(E)}^{\B(E)}$ has enough projective objects by Corollary \ref{perfect}. The description of bialgebra cohomology as an Ext functor \cite{tai04,tai04b} now shows that the bialgebra cohomology
of $\B(E)$ is given by the cohomology of the complex obtained by applying the functor 
${\rm Hom}_{\yd_{\B(E)}^{\B(E)}}(-,\C)$ to any projective resolution of the trivial Yetter-Drinfeld module $\C$.
Thus, using the resolution of Theorem \ref{reso}, the description of $H_b^*(\B(E))$ is a direct computation, that we leave to the reader. 
\end{proof}

\subsection{$L^2$-Betti numbers.} Let $A$ be compact Hopf algebra, i.e. $A$ is the Hopf $*$-algebra of polynomial functions on  a compact quantum group \cite{wo}, and assume that $A$ is of Kac type (the Haar state of $A$ is tracial, or equivalently the square of the antipode is the identity). The $L^2$-Betti numbers of $A$ have been defined in \cite{ky08}. There are several possible equivalent definitions \cite{cosh,ky08, thom} and the one we shall use is
$$\beta_k^{(2)}(A)=\dim_{\mathcal M^{\rm op}}{\rm Tor}^A_k(\C_\varepsilon, \  _A\!\mathcal M)$$ 
where $\mathcal M$ is the (finite) von Neumann algebra of $A$, the left $A$-module structure on $_A\!\mathcal M=\mathcal M$ is given by left multiplication through the natural inclusion $A \subset \mathcal M$, and $\dim_{\mathcal M^{\rm op}}$ is L\"uck's  dimension function for modules over finite von Neuman algebras \cite{luck}.

Now let $F \in \GL_n(\C)$ with $F\overline{F} \in \mathbb R I_n$. Recall \cite{ba96} that
$A_o(F)$ is the universal $*$-algebra with generators $(u_{ij})$, $1\leq i,j \leq n$ and relations making the matrix $u=(u_{ij})$ unitary and $u=F \overline{u}F^{-1}$. The $*$-algebra $A_o(F)$ has a natural Hopf $*$-algebra structure (it is isomorphic, as a Hopf algebra, to the previous $\B((F^{t})^{-1})$) and is a compact Hopf algebra. Moreover $A_o(F)$ is of Kac type if and only if $F^t=\pm F$. Thus in this case 
$$A_o(F) \simeq A_o(I_n)=A_o(n), \quad {\rm or} \quad A_o(F) \simeq A_o(J_{2m})$$
where for $2m=n$, $J_{2m} \in \GL_{2m}(\mathbb C)$ is the anti-symmetric matrix    
$$J_{2m}= \left(\begin{array}{cc} 0_m & I_m \\
                          -I_m & 0_m\\
       \end{array} \right)$$
Collins, H\"artel and Thom have shown, combining the results in \cite{cht} with the vanishing of $\beta_1^{(2)}(A_o(n))$ by Vergnioux \cite{ver12},  that all the $L^2$-Betti numbers $\beta_k^{(2)}(A_o(n))$ vanish. 
Similar arguments lead to the following result, which completes the computation of the $L^2$-Betti numbers
of the $A_o(F)$'s of Kac type.

\begin{theorem}
 For any $m\geq 1$ and $k\geq 0$, we have $\beta_k^{(2)}(A_o(J_{2m}))=0$.
\end{theorem}

\begin{proof}
For $m=1$ it is already known \cite{ky08}
that the $L^2$-Betti numbers of $A_o(J_{2})$ all vanish, since $A_o(J_{2})$ is commutative. So we assume that $m \geq 2$.
First, by Theorem \ref{reso},  we have $\beta_k^{(2)}(A_o(J_{2m}))=0$ for $k > 3$.  We have
$\beta_0^{(2)}(A_o(J_{2m}))=0$ by \cite{kye11} and hence by Poincar\'e duality and the fact that the $L^2$-Betti numbers can be defined in terms of $L^2$-cohomology \cite{thom} we have $\beta_3^{(2)}(A_o(J_{2m}))=0=\beta_0^{(2)}(A_o(J_{2m}))$. Similarly by Poincar\'e duality we have $\beta_2^{(2)}(A_o(J_{2m}))=\beta_1^{(2)}(A_o(J_{2m}))$. By
\cite{ver12}, Theorem 4.4 and the proof of Corollary 5.2 (the proof of Corollary 5.2 in \cite{ver12} is valid for $A_o(J_{2m})$ since it has the property of rapid decay \cite{ver07}), we have $\beta_1^{(2)}(A_o(J_{2m}))=0$, and we are done.
\end{proof}

\section{Conclusion}

We have shown that there might exist strong links between the Hochschild (co)homologies of Hopf algebras that have equivalent tensor categories of comodules, although the ring-theoretical properties of the underlying algebras might be very different. We cannot expect to have functoriality at the level of the computation of Hochschild (co)homology group, the situation is rather that if one of the Hopf algebras has a very special homological feature (a free Yetter-Drinfeld resolution of the counit), then so has the other.  

A similar situation had been observed in the work of Voigt \cite{voi11} on the $K$-theory of free orthogonal quantum groups: the existence of a tensor category equivalence does not seem to imply functoriality at the level of $K$-theory groups, but is enough to ensure that one can transport a special homological situation, namely the validity of the Baum-Connes conjecture. 

We hope that the present paper will bring further evidence to convince the reader that tensor category methods can be useful tools in the homological study of Hopf ($C^*$-)algebras.

\end{document}